\documentclass[11pt]{amsart}

\usepackage{amsfonts}
\usepackage{amssymb}
\usepackage{hyperref}

\theoremstyle{plain}
\newtheorem{thm}{Theorem}[section]
\newtheorem{lem}[thm]{Lemma}
\newtheorem{cor}[thm]{Corollary}
\newtheorem{prop}[thm]{Proposition}

\theoremstyle{definition}
\newtheorem{defi}[thm]{Definition}
      
\theoremstyle{remark}

\newcommand{\lemref}[1]{\hyperref[#1]{Lemma \ref*{#1}}}
\newcommand{\thmref}[1]{\hyperref[#1]{Theorem \ref*{#1}}}
\newcommand{\propref}[1]{\hyperref[#1]{Proposition \ref*{#1}}}
\newcommand{\corref}[1]{\hyperref[#1]{Corollary \ref*{#1}}}
\newcommand{\defref}[1]{\hyperref[#1]{Definition \ref*{#1}}}
\newcommand{\remref}[1]{\hyperref[#1]{Remark \ref*{#1}}}
\newcommand{\conjref}[1]{\hyperref[#1]{Conjecture \ref*{#1}}}

\newcommand{\Ker}{\mathrm{Ker}}

\makeatletter
\newcommand*{\defeq}{\mathrel{\rlap{%
                     \raisebox{0.27ex}{$\m@th\cdot$}}%
                     \raisebox{-0.27ex}{$\m@th\cdot$}}%
                     =}

\numberwithin{equation}{section}
\makeatother

\makeatletter
\def\@setcopyright{}
\def\serieslogo@{}
\makeatother

\begin{document}

\title{Gradients of sequences of subgroups in a direct product}

\author{Nikolay Nikolov}
\address{University of Oxford, OX2 6GG Oxford, UK}
\email{nikolay.nikolov@maths.ox.ac.uk‏}

\author{Zvi Shemtov}
\address{Raymond and Beverly Sackler School of Mathematical Sciences, Tel-Aviv University, Tel-Aviv, Israel}
\email{s.tzvika@gmail.com}

\author{Mark Shusterman}
\address{Raymond and Beverly Sackler School of Mathematical Sciences, Tel-Aviv University, Tel-Aviv, Israel}
\email{markshus@mail.tau.ac.il} 
\date{}
\begin{abstract}

For a sequence $\{U_n\}_{n = 1}^\infty$ of finite index subgroups of a direct product $G \defeq A \times B$ of finitely generated groups, we show that 
$$\lim_{n \to \infty} \frac{\min\{|X| : \langle X \rangle = U_n\}}{[G : U_n]} = 0$$ 
once $[A : A \cap U_n], [B : B \cap U_n] \to \infty$ as $n \to \infty$.
Our proof relies on the classification of finite simple groups.
For $A,B$ that are finitely presented we show that
$$ \lim_{n \to \infty} \frac{\log |\mathrm{Torsion}(U_n^{\mathrm{ab}})|}{[G : U_n]} = 0. $$

\end{abstract}

\maketitle

\section{Introduction}

With motivation coming from the theory of $3$-manifolds, 
Lackenby defined in \cite{L} the rank gradient of a sequence $\{U_n\}_{n=1}^\infty$ of finite index subgroups of a finitely generated group $G$ to be the following combinatorial invariant
\begin{equation} \label{DefRGEq}
\inf_{n}\frac{d(U_n)-1}{[G:U_n]}
\end{equation}
where $d(K)$ is the least cardinality of a generating set for the group $K$.
As can be seen from \cite{AGN, AJN, AN, AG, Gi, GK, KN, Lack2, Loh, O, P, Sch, Sh0, Sh1, Sh2} the rank gradient has been extensively studied, calculated in various cases, and related to the notion of cost from ergodic theory.
Furthermore, the rank gradient controls other interesting combinatorial invariants such as the $p$-gradient,
and the rate of growth of the first Betti number in finite index subgroups, that were studied, for instance, in \cite{7s, AJN, BLLS, CE, CW, EL, KS, L2, LLS, Lu, Lu2, LuO, P, Sau}.  
It is one of the main goals of asymptotic and measured group theory to calculate these gradients,
and if possible, to show that they coincide for various sequences, 
and equal to the appropriate analytic and ergodic invariants (such as the first $\ell^2$-Betti number and the cost of $G$).
Here we accomplish this task for all the aforementioned gradients and all sequences in case that $G$ is a direct product.

\begin{thm} \label{Res}
Let $A,B$ be finitely generated groups, 
set $G \defeq A \times B$, 
and let $\{U_n\}_{n=1}^\infty$ be a sequence of finite index subgroups of $G$.
Assume that 
\begin{equation} \label{IndAs1}
[A : A \cap U_n], [B : B \cap U_n] \to \infty
\end{equation}
as $n \to \infty$. 
Then 
\begin{equation} \label{aimRes}
\lim_{n \to \infty}\frac{d(U_n)-1}{[G:U_n]} = 0.
\end{equation}
\end{thm}

We should note that even the vanishing of the $p$-gradient (or the Betti gradient) was not previously known for these sequences.

An open problem in ergodic theory analogous to \thmref{Res} is to determine whether the direct product of every pair of finitely generated groups has fixed price. 
Indeed, the special case of \thmref{Res} considering normal (or even Farber) chains (nested sequences), 
is at the focus of \cite{ST}. 
Another special case that was already known is when $A,B$ are torsion-free 
(and thus $G$ is right-angled, see \cite{AGN, KM}) and the sequence is Farber.

Additional motivation for \thmref{Res} comes directly from (arithmetic) $3$-manifolds.
As recent breakthroughs by Agol, Wise (and others) tell us that many $3$-manifolds virtually fiber over the circle,
we know that their fundamental groups are semidirect products (up to a finite index).
Calculating the rank gradient (or the $p$-gradient) of semidirect products is both notoriously difficult,
and interesting from a topological and number-theoretic point of view,
so it makes sense to examine the special case of direct products first.

Our assumption \eqref{IndAs1} is necessary, since otherwise the gradients may become positive once this is the case for one of the factors $A,B$.
Indeed the sequences we consider are very general, and this allows us to obtain vanishing of gradients for Farber chains in groups (virtually) generated by a pair of commuting infinite subgroups.
These are the groups `presentable by a product' that were studied in \cite{D, DK, KL0, KL}.

Having considered the Betti gradient (measuring the free part of abelianizations in sequences), 
it is tempting to study the torsion gradient
\begin{equation} \label{DefTGEq}
\liminf_{n \to \infty} \frac{\log |\mathrm{Torsion}(U_n^{\mathrm{ab}})|}{[G : U_n]}
\end{equation}
for which we assume that $G$ is finitely presented (otherwise, \eqref{DefTGEq} may be infinite).
This gradient, having connections to number theory, geometry, and topology,
became an object of intensive study, as witnessed by 
\cite{AGN, BGS, BL, BV, BVS, Braun, BD, CV, KKN, Lu, Lu2, Nik, Raim, Riv, Seng, SW, Sun}.

\begin{thm} \label{SecRes}
Let $A,B$ be finitely presented groups, 
set $G \defeq A \times B$, 
and let $\{U_n\}_{n=1}^\infty$ be a sequence of finite index subgroups of $G$.
Assume that 
\begin{equation} \label{IndAs2}
[A : A \cap U_n], [B : B \cap U_n] \to \infty
\end{equation}
as $n \to \infty$. 
Then 
\begin{equation} \label{aimRes}
\frac{\log |\mathrm{Torsion}(U_n^{\mathrm{ab}})|}{[G : U_n]} = 0.
\end{equation}
\end{thm}

\thmref{SecRes} provides the first large family of groups for which the torsion gradient vanishes in (almost) all sequences, including some that are not even Farber.

\section{Preliminaries}
Let $G$ be a finitely generated group 
and let $H\leq G$ be a finite index subgroup of $G$.
We will be frequently using the bound
\begin{equation} \label{SBEq}
d(H) \leq d(G)[G:H].
\end{equation}
that follows from \cite[Theorem 11.44]{R}. 

\begin{prop} \label{NormalityProp}
Let $A,B$ be groups, set $G \defeq A \times B$ and let $K \leq G$ be a subgroup for which $BK = G$.
Then $A \cap K \lhd A$.
\end{prop}

\begin{proof}
Take $a \in A, \ x \in A \cap K$ and note that since $A \subseteq BK$, there exist $b \in B, \ k \in K$ 
such that $a = bk$. As $A \lhd G$ we see that $kxk^{-1} \in A \cap K$, and $[A,B] = \{1\}$
implies that $axa^{-1} = bkxk^{-1}b^{-1} = kxk^{-1} \in A \cap K$.
\end{proof}

\begin{defi} \label{NormalDef}
Let $G$ be a group and let $S,X \subseteq G$ be subsets.
We denote by $\langle S \rangle^X$ the subgroup of $G$ generated by the conjugates of $S$ by elements of $X$.
For a normal subgroup $N \lhd G$ we define $d_G(N)$ to be the least cardinality of a subset $S \subseteq N$ for which $\langle S \rangle^G = N$.
\end{defi}

\begin{defi} \label{RelsDef}
Let $K$ be a finite group, let $d \in \mathbb{N}$, and let $T=\{t_1,\dots,t_d\}$ be a generating multiset of $K$.
Take a free group $F$ on $X = \{x_1,\dots,x_d\}$ and let $\varphi \colon F \to K$ be the unique surjection with $\varphi(x_i) = t_i$ for $1 \leq i \leq d$.
We set $r(K,T) \defeq d_F(\Ker(\varphi))$ and think of this quantity as the least number of relations needed to present $K$ using the generating multiset $T$.  
\end{defi}

\begin{prop} \label{connProp}
Let $G$ be a group generated by a finite subset $T \subseteq G$,
let $N\lhd G$ be a normal subgroup of finite index in $G$,
and define a multiset $\overline{T} \subseteq G/N$ by $\overline{T} \defeq \{tN \ | \ t \in T\}$.
Then $d_G(N) \leq r(G/N,\overline{T})$.
\end{prop}

\begin{proof}
Let $F$ be the free group on $T$, 
and let $\varphi \colon F \to G/N$ be the unique surjection for which $\varphi(t) = tN$ for all $t \in T$.
By \ref{RelsDef} there exists a subset $S \subseteq \Ker(\varphi)$ of cardinality $r(G/N,\overline{T})$ such that 
$\langle S \rangle^F = \Ker(\varphi)$.
Define a surjection $\psi \colon F \to G$ by $\psi|_T = \mathrm{Id}_T$, and observe that
\begin{equation}
N = \psi(\Ker(\varphi)) = \psi(\langle S \rangle^F) = \langle \psi(S) \rangle^G
\end{equation}
so that $d_G(N) \stackrel{\ref{NormalDef}}{\leq} |\psi(S)| \leq |S| = r(G/N,\overline{T})$ as required.
\end{proof}

Let $K$ be a finite group generated by a multiset $T=\{t_1,\dots,t_d\}$, let $N$ be a minimal normal subgroup of $K$,
let $S$ be a finite simple group onto which $N$ surjects, and define the multiset $\overline{T} \defeq \{t_1N,\dots,t_dN\}$.
The following bound comes from the argument appearing in the proof of Theorem 1 in \cite{M} 
and in the discussion following the proof of Theorem 2 therein. 
\begin{equation} \label{AvinoamBoundEq}
r(K,T) \leq r(K/N,\overline{T})+6d\log_2{|N|}+ r(S,W)
\end{equation}
where $W$ is any pair of elements generating $S$.
It follows from \cite[Corollary A', Lemma 2.1, Theorem 4.34]{GKKL} and \cite[Theorem 1]{M} that
\begin{equation} \label{SimpleBoundEq}
r(S,W) \leq 8|S|^{3/7}.
\end{equation} 

\begin{cor} \label{mainCor}
Let $K$ be a finite group generated by a multiset $T$.
Then
\begin{equation}
r(K,T) \leq 128|T||K|^{3/7}.
\end{equation} 

\end{cor}

\begin{proof}

Let $N$ be a minimal normal subgroup of $K$,
and let $S$ be a finite simple factor of $N$.
By induction, \eqref{AvinoamBoundEq}, and \eqref{SimpleBoundEq}, we have
\begin{equation}
\begin{split}
r(K,T) &\leq r(K/N,\overline{T}) + 6|T|\log_2{|N|}+ 8|S|^{3/7} \\
&\leq 128|T|\Big(\frac{|K|}{|N|}\Big)^{3/7} + 6|T|\log_2{|N|} + 8|N|^{3/7} \\
&\leq 128|T|\Big(\frac{|K|}{|N|}\Big)^{3/7} + 24|T||N|^{3/7} + 8|T||N|^{3/7} \\
&= 128|T| \Bigg( \Big(\frac{|K|}{|N|}\Big)^{3/7} + |N|^{3/7} - \frac{3}{4}|N|^{3/7} \Bigg) \\
&\leq 128|T| \Big( |K|^{3/7} + 1 - \frac{3}{4}|N|^{3/7} \Big) \\
&\leq 128|T| \Big( |K|^{3/7} + 1 - \frac{3}{4}2^{3/7} \Big) \leq 128|T||K|^{3/7}.
\end{split}
\end{equation}
\end{proof}

Let $E$ be a finite group with a presentation
\begin{equation}
1 \longrightarrow N \longrightarrow F \longrightarrow E \longrightarrow 1
\end{equation}
where $F$ is free.
The Schur multiplier (see \cite[Theorem 10.12]{Rot}) of $E$ is
\begin{equation}
M(E) = \big(N \cap [F,F]\big) / [F,N].
\end{equation}
This is a finite abelian group whose order admits the following bound.
\begin{cor}
Let $E$ be a finite group. Then
\begin{equation}
|M(E)| \leq |E|^{\log |E|}.
\end{equation}

\end{cor}

\begin{proof}

Let $S$ be the set of primes dividing $|E|$, and for each $p \in S$ fix a $p$-Sylow subgroup $P$ of $E$.
It follows from \cite[Proposition 2.1.1]{Karp} that
\begin{equation} \label{decprimEq}
|M(E)| = \prod_{p \in S} |M(E) \otimes_\mathbb{Z} \mathbb{Z}_p|.
\end{equation}
By \cite[Theorem 2.1.2]{Karp},
for every $p \in S$ we have
\begin{equation} \label{SylowBoundEq}
|M(E) \otimes_\mathbb{Z} \mathbb{Z}_p| \leq |M(P)|
\end{equation}
and by \cite[Corollary 3.1.5]{Karp}
\begin{equation} \label{GreenEq}
|M(P)| \leq |P|^{\log |P|}.
\end{equation}
Combining everything, we get that
\begin{equation}
\begin{split}
|M(E)| &\stackrel{\ref{decprimEq}}{\leq} \prod_{p \in S} |M(E) \otimes_\mathbb{Z} \mathbb{Z}_p| 
\stackrel{\ref{SylowBoundEq}}{\leq} \prod_{p \in S} |M(P)| \\
&\stackrel{\ref{GreenEq}}{\leq} \prod_{p \in S} |P|^{\log |P|} \leq \prod_{p \in S} |P|^{\log |E|} = |E|^{\log |E|}.
\end{split}
\end{equation}
\end{proof}

In terms of our presentation for $E$ we thus have
\begin{equation}
\begin{split}
\Big[ [F,F] : [F,N] \Big] &\leq \Big[ [F,F] : N \cap [F,F] \Big] \cdot \Big[ N \cap [F,F] : [F,N] \Big] \\
&\leq [F : N] \cdot |M(E)| \leq |E|^{1 + \log |E|}.
\end{split}
\end{equation}
In a manner similar to \propref{connProp} one can deduce the following.
\begin{cor} \label{MultipCor}
For a finite index normal subgroup $A_0$ of a group $A$ we have
\begin{equation}
\Big[ [A,A] : [A,A_0] \Big] \leq [A : A_0]^{1 + \log [A : A_0]}.
\end{equation} 

\end{cor}

\section{Upper bounds on the number of generators}

We establish several bounds on the number of generators of a finite index subgroup of a direct product, 
and conclude that the rank gradient vanishes.

\begin{thm} \label{TechResThm}

Let $A,B$ be finitely generated groups,
and let $H$ be a finite index subgroup of $G \defeq A \times B$.
Then $d(H)$ is bounded by:

\begin{enumerate}

\item $d(G)([G : AH] + [AH : H])$

\item $d(G)([G : BH] + [BH : H])$

\item $d(G)([G : AH] + 130[G : BH][G : H]^{3/7})$.

\end{enumerate}

\end{thm}

\begin{proof}

For $(1)$ note that
\begin{equation} \label{ToGet1Eq}
\begin{split}
d(H) &\leq d(H/H \cap A) + d(H \cap A) \stackrel{\ref{SBEq}}{\leq} d(AH/A) + d(A)[A : H \cap A] \\
&\leq d(AH) + d(A)[AH : H] \stackrel{\ref{SBEq}}{\leq} d(G)[G : AH] + d(G)[AH : H]
\end{split}
\end{equation}
as required. To get $(2)$ just replace $A$ with $B$ in \eqref{ToGet1Eq}.

For $(3)$ let $\pi_A, \pi_B$ be the projections from $G$ onto $A,B$
and observe that $H \leq \pi_A(H) \times \pi_B(H)$ is a subgroup that complements $\pi_B(H)$ 
(that is, $\pi_B(H)H = \pi_A(H) \times \pi_B(H)$).
By \propref{NormalityProp}, $\pi_A(H) \cap H \lhd \pi_A(H)$.

We can thus take a subset $S \subseteq \pi_A(H) \cap H$ of least cardinality, with
\begin{equation} \label{SDefEq}
\langle S \rangle^{\pi_A(H)} = \pi_A(H) \cap H.
\end{equation}
Furthermore, take $R_A$ (respectively, $R_B$) to be a subset of $H$ mapped bijectively by $\pi_A$ (respectively, $\pi_B$) onto a generating set of $\pi_A(H)$ (respectively, $\pi_B(H)$) of least cardinality.  
Set $L \defeq \langle S \cup R_A \cup R_B \rangle$ and note that $L \leq H$.
Since $S \subseteq A$, conjugation by $B$ does not affect it, so we find that 
\begin{equation} \label{ContKerEq}
\begin{split}
\Ker(\pi_B|_{H}) &= \pi_A(H) \cap H \stackrel{\ref{SDefEq}}{=} \langle S \rangle^{\pi_A(H)} \\
&= \langle S \rangle^{B\pi_A(H)} = \langle S \rangle^{B\langle R_A \rangle} = \langle S \rangle^{\langle R_A \rangle} \stackrel{\ref{NormalDef}}{\leq} L.
\end{split}
\end{equation}
On the other hand, $\pi_B(H) = \pi_B(\langle R_B \rangle) \leq \pi_B(L)$ so in conjunction with \eqref{ContKerEq} we conclude that $L = H$.
Thus
\begin{equation} \label{3BoundEq}
\begin{split}
d(H) &= d(L) \leq |S| + |R_A| + |R_B| \\
&=  d_{\pi_A(H)}(\pi_A(H) \cap H) + d(\pi_A(H)) + d(\pi_B(H)) \\
&\stackrel{\ref{connProp}}{\leq} r(\pi_A(H)/\pi_A(H) \cap H, \pi_A(R_A)) + d(\pi_A(H)) + d(\pi_B(H)) \\
&\stackrel{\ref{mainCor}}{\leq} 128d(\pi_A(H))[G : H]^{3/7} + d(\pi_A(H)) + d(\pi_B(H)).
\end{split}
\end{equation}
Moreover,
\begin{equation} \label{32BoundEq}
d(\pi_A(H)) \stackrel{\ref{SBEq}}{\leq} d(A)[A : \pi_A(H)] = d(A)[G : BH] \leq d(G)[G : BH]
\end{equation}
and similarly, we have $d(\pi_B(H)) \leq d(G)[G : AH]$. Combining this inequality, \eqref{3BoundEq}, and \eqref{32BoundEq} we obtain $(3)$.
\end{proof}

Let us now deduce \thmref{Res}.

\begin{proof}
Suppose that our claim is false, 
so (after passing to a subsequence) we may assume that the limit in \eqref{aimRes} is positive.
By \thmref{TechResThm} $(1)$,
\begin{equation} \label{CorBoundEq}
\frac{d(U_n)}{[G : U_n]} \leq d(G)\Big(\frac{1}{[AU_n : U_n]} + \frac{1}{[G : AU_n]}\Big)
\end{equation}
and the first summand on the right hand side tends to $0$ in view of our assumption that $[A : A \cap U_n] \to \infty$ as $n \to \infty$. 
Since the left hand side of \eqref{CorBoundEq} tends to some $c > 0$,
we conclude that $[G : AU_n]$ is bounded as $n \to \infty$.
Similarly, $[G : BU_n]$ is bounded.
Finally, apply \thmref{TechResThm} $(3)$ to $U_n$. 
\end{proof}

\section{Lower bounds on the number of generators}

To which extent are the bounds in \thmref{TechResThm} tight?
Suppose that $A$ and $B$ are isomorphic to a free group $F$ on two generators.
Clearly, $(1)$ and $(2)$ are tight up to a constant once $H \defeq A_0 \times B_0$ where $A_0, B_0 \leq A,B$ are subgroups with $[A : A_0] = [B : B_0]$.

It is conjectured (\cite[Conjecture 2]{M}) that every finite group has a presentation with a logarithmic number of relations. 
If this improvement of \corref{mainCor} holds, then the argument from the proof of \thmref{TechResThm} $(3)$ gives a logarithmic bound on the number of generators as a function of the index of $H$ in $G$.
Let us show that such a bound is tight (up to a constant).
Fix a prime $p$, and let $\{P_n\}_{n=1}^{\infty}$ be the lower $p$-central series defined by
\begin{equation}
P_1=F, \ P_{n+1}= P_n^p[F,P_n].
\end{equation}
Set
\begin{equation}
p^{a_n} \defeq |P_n/P_{n+1}|, \quad b_n \defeq \sum_{i=1}^n a_i.
\end{equation} 
By \cite[Corollary 3.4]{HM}, $a_n= \sum_{i=1}^n r_i$ where $r_i$ are the Witt numbers given by
\begin{equation}
r_i= \frac{1}{i} \sum_{j |i} \mu \left(\frac{i}{j} \right) 2^j.
\end{equation}
It is thus easy to see that $r_n \sim 2^n/n$, and in particular the radius of convergence of $\sum_{n=1}^ \infty r_nx^n$ is $1/2$. 
Multiplying by $(\sum_{n=0}^\infty x^n)^2$ we deduce that $\sum_{n=1}^ \infty b_nx^n$ has the same radius of convergence,
and thus 
\begin{equation} \label{limEq}
\limsup_{n \to \infty} \frac{b_n}{b_{n-1}} \geq 2.
\end{equation}

Set $U_n \defeq \{(x,y) \in F \times F \ |  \ xP_n=yP_n \}$. 
We have
\begin{equation}
[F \times F : U_n] = p^{b_{n-1}} = [ F \times \{1\} : F\times \{1\} \cap U_n].
\end{equation}
We claim that $U_n$ maps onto the elementary abelian group $P_n/P_{n+1}$.
Let $\Delta \leq (F/P_{n+1}) \times (F/P_{n+1})$ be the image of $U_n$ mod $P_{n+1}$. 
That is 
\begin{equation}
\Delta \defeq \{  (xP_{n+1}, yP_{n+1}) \in (F/P_{n+1}) \times (F/P_{n+1}) \ | \ xP_n=yP_n \}.
\end{equation}
Let $L \defeq \{ (xP_{n+1},xP_{n+1}) \ | \ x \in F \}$ be the diagonal subgroup of $\Delta$.
Clearly, $L$ is a normal subgroup of $\Delta$ that commutes with $(P_n/P_{n+1}) \times (P_n/P_{n+1})$. 
Moreover, $L \cap \big((P_n/P_{n+1}) \times (P_n/P_{n+1})\big)$ is the diagonal subgroup isomorphic to $P_n/P_{n+1}$. 
It follows that 
\begin{equation}
\frac{\Delta}{L} \cong \frac{(P_n/P_{n+1}) \times (P_n/P_{n+1})}{L \cap ((P_n/P_{n+1}) \times (P_n/P_{n+1}))} \cong P_n/P_{n+1}.
\end{equation}
Hence $U_n$ has $P_n/P_{n+1}$ as a homomorphic image and therefore $d(U_n) \geq a_n$.
At last, note that for any $\epsilon > 0$ we have
\begin{equation} \label{LastEq}
\frac{d(U_n)}{\log_p [F \times F : U_n]} \geq \frac{a_n}{b_{n-1}}= \frac{b_{n}}{b_{n-1}}-1 \stackrel{\ref{limEq}}{\geq} 1-\epsilon
\end{equation}
where the last inequality holds for infinitely many values of $n$ (that is, \eqref{LastEq} holds for a subsequence of $\{U_n\}_{n=1}^{\infty}$).

\section{Upper bounds on torsion in the abelianization}

In the following we give the bound needed to establish \thmref{SecRes}.
The torsion subgroup of an abelian group $M$ is denoted by $t(M)$.

\begin{lem} \label{TorLem}

Let $A,B$ be finitely generated groups,
and let $H$ be a finite index subgroup of $G \defeq A \times B$.
Then
\begin{equation} \label{TorBoubdEq}
|t(H^{\mathrm{ab}})| \leq 
|t(\pi_A(H)^{\mathrm{ab}})| \cdot |t(\pi_B(H)^{\mathrm{ab}})| \cdot [G : H]^{2(1 + \log [G : H])}.
\end{equation} 

\end{lem}

\begin{proof}
We use the shorthand $K'$ to denote the commutator subgroup $[K,K]$ of a group $K$, and set
\begin{equation} \label{TorspDefEq}
A_0 \defeq \pi_A(H) \cap H = A \cap H, \quad B_0 \defeq \pi_B(H) \cap H = B \cap H.
\end{equation} 
Applying \propref{NormalityProp}, we conclude that $A_0 \lhd \pi_A(H), \ B_0 \lhd \pi_B(H)$ are normal subgroups of finite index.
One shows easily that
\begin{equation} \label{EasEq}
[\pi_A(H), A_0] \times [\pi_B(H), B_0] \leq H' \leq H \cap \big(\pi_A(H)' \times \pi_B(H)' \big)
\end{equation}
so by \corref{MultipCor} we have
\begin{equation} \label{LTorpEq}
\begin{split}
&\Big[H \cap \big(\pi_A(H)' \times \pi_B(H)' \big) : H'\Big] \leq \\
&\Big[\pi_A(H)' \times \pi_B(H)' : H'\Big] \stackrel{\ref{EasEq}}{\leq} \\
&\Big[\pi_A(H)' \times \pi_B(H)' : [\pi_A(H), A_0] \times [\pi_B(H), B_0]\Big] = \\
&\Big[\pi_A(H)' : [\pi_A(H), A_0] \Big] \cdot \Big[\pi_B(H)' : [\pi_B(H), B_0]\Big] \stackrel{\ref{MultipCor}}{\leq} \\
&\Big[\pi_A(H) : A_0\Big]^{1 + \log [\pi_A(H) : A_0]} \cdot \Big[\pi_B(H) : B_0\Big]^{1 + \log [\pi_B(H) : B_0]} 
\stackrel{\ref{TorspDefEq}}{\leq} \\
&\Big[G : H \Big]^{2(1 + \log[G : H])}.
\end{split}
\end{equation}
In particular, by \eqref{LTorpEq}, the subgroup 
\begin{equation}
M \defeq \Big( H \cap \big(\pi_A(H)' \times \pi_B(H)' \big) \Big) / H'
\end{equation}
of $H^{\mathrm{ab}}$ is finite, so
\begin{equation}
\begin{split}
|t(H/H')| &= |M| \cdot |H^{\mathrm{ab}}/M| \leq |M| \cdot t\Big(\big(\pi_A(H) \times \pi_B(H)\big)^{\mathrm{ab}}\Big) \\
&\stackrel{\ref{LTorpEq}}{\leq} [G : H]^{2(1 + \log [G : H])} \cdot |t(\pi_A(H)^{\mathrm{ab}})| \cdot |t(\pi_B(H)^{\mathrm{ab}})|.
\end{split} 
\end{equation}
\end{proof}

In order to obtain \thmref{SecRes} from \lemref{TorLem}, one just has to note that our assumption \eqref{IndAs2} on the growth of the indices to infinity in \thmref{SecRes}, implies that the first two factors in \eqref{TorBoubdEq} grow subexponentially with the index of a subgroup in a sequence, since the torsion in the abelianization of finite index subgroups of a finitely presented group grows at most exponentially with the index. 

\section*{Acknowledgments}
This work is partially supported by the ERC grant ANALYTIC no. 259527 of Goulnara Arzhantseva.
Mark Shusterman is grateful to the Azrieli Foundation for the award of an Azrieli Fellowship.
The third author was partially supported by a grant of the Israel Science Foundation with cooperation of UGC no. 40/14.

\end{document}